\documentclass[12pt,reqno]{amsart}
\usepackage[utf8]{inputenc}
\usepackage{amsfonts, amssymb, amsmath, amsthm, mathtools}
\usepackage{mathrsfs} 
\usepackage{latexsym}
\usepackage{enumerate}
\usepackage{multicol}
\usepackage{times}
\usepackage{verbatim}
\usepackage{tikz-cd}
\usepackage{fullpage}
\usepackage{here}
\usepackage{url}
\usepackage{csquotes}
\usepackage{placeins}
\usepackage{epic}
\usepackage{graphicx}
\usepackage{epstopdf}
\usepackage{pstricks}
\usepackage{pst-plot}
\usepackage{tikz}
\usetikzlibrary{shapes.geometric, positioning, calc}
\input xypic
\xyoption{all}
\xyoption{poly}
\usepackage[colorlinks=true]{hyperref}
\hypersetup{citecolor=blue, linkcolor=blue}
\usepackage[capitalise]{cleveref}

\crefname{equation}{}{}
\crefname{figure}{{\sc Figure}}{{\sc Figure}}
\crefname{subsection}{Subsection}{Subsections}

\usetikzlibrary{matrix,arrows,decorations.pathmorphing}
\usepackage[euler]{textgreek}
\usepackage{tikz,tikz-cd}
\usepackage[colorinlistoftodos, textwidth = 2.3cm]{todonotes}

\newtheorem{theorem}{Theorem}[section]
\newtheorem{proposition}[theorem]{Proposition}
\newtheorem{lemma}[theorem]{Lemma}
\newtheorem{corollary}[theorem]{Corollary}

\newtheorem{claim}[theorem]{Claim}

\newtheorem*{claim*}{Claim}

\theoremstyle{definition}

\newtheorem{definition}[theorem]{Definition}
\newtheorem{remark}[theorem]{Remark}

\newcommand{\F}{{\mathbb F}}
\newcommand{\Z}{{\mathbb Z}}
\newcommand{\Tr}{\operatorname{Tr}}

\usepackage{tabstackengine}
\stackMath

\numberwithin{equation}{section} 
\numberwithin{figure}{section}
\numberwithin{table}{section}

\title{Paley-like quasi-random graphs arising from polynomials}
\author{Seoyoung Kim}
\address{Departement Mathematik und Informatik, Universit\"at Basel, Spiegelgasse 1, 4051 Basel, Switzerland}
\email{seoyoung.kim@unibas.ch}
\author{Chi Hoi Yip}
\address{School of Mathematics\\ Georgia Institute of Technology\\Atlanta, GA 30332\\ United States}
\email{cyip30@gatech.edu}
\author{Semin Yoo}
\address{Discrete Mathematics Group \\ Institute for Basic Science \\ 55 Expo-ro Yuseong-gu, Daejeon 34126 \\ South Korea}
\email{syoo19@ibs.re.kr}

\begin{document}

\begin{abstract}
Paley graphs and Paley sum graphs are classical examples of quasi-random graphs. In this paper, we provide new constructions of families of quasi-random graphs that behave like Paley graphs but are neither Cayley graphs nor Cayley sum graphs. These graphs give a unified perspective of studying various graphs arising from polynomials over finite fields, such as Paley graphs, Paley sum graphs, and graphs arising from Diophantine tuples and their generalizations. 
We also obtain lower bounds on the clique and independence numbers of the graphs in these families. 
\end{abstract}

\maketitle

\section{Introduction}\label{sec: intro}
Throughout this paper, we let $q$ be an odd prime power,  $\F_q$ the finite field with $q$ elements, and $\F_q^*=\F_q \setminus \{0\}$. When $q \equiv 1 \pmod 4$, the {\em Paley graph} $P_q$ is the graph whose vertices are the elements of $\F_q$, such that two distinct vertices are adjacent if and only if their difference is a square in $\F_q$. {\em Paley sum graphs} are defined similarly, where two distinct vertices are adjacent when their sum is a square in $\F_q$. Paley graphs and Paley sum graphs are classical examples of quasi-random graphs \cite{C91}. The main purpose of this paper is to construct new Paley-like graphs using polynomials defined over finite fields and to show their quasi-random property. 
We also give a nontrivial lower bound on the clique number and independence number of the graphs.

Let $f \in \F_q[x,y]$ be a polynomial. We define a graph $X_{f,q}$ on $\F_q$, such that two distinct vertices $a,b\in \F_q$ are adjacent if and only if $f(a,b)$ is a square in $\F_q$.
To ensure that the graph $X_{f,q}$ is undirected, we assume that for $a,b \in \F_q$, $f(a,b)$ is a square if and only if $f(b,a)$ is a square. Furthermore, some conditions on the polynomial $f$ have to be imposed so that $X_{f,q}$ could behave like Paley graphs. 
At first glance, the only assumption needed for defining Paley-like graph $X_{f,q}$ is that $f(x,y)$ is not a constant multiple of a square of a polynomial; otherwise, the graph $X_{f,q}$ is complete or almost empty. However, a more refined assumption is required, and we introduce the following terminology. 
Following Gyarmati and S\'{a}rk\"{o}zy \cite[Proposition 1]{GS08}, we write 
\[f(x,y)=F(x)G(y)H(x,y),
\]
where $H(x,y)$ is the \emph{primitive kernel} of $f(x,y)$ (see \cref{sec: prelim}). It turns out that if $H(x,y)$ is a constant multiple of a square of a polynomial but $f(x,y)$ itself is not, then the graph $X_{f,q}$ still has a simple structure; see \cref{lem:natural}. 
This observation motivates the following definition of admissible polynomials.
\begin{definition}
Given $f \in \F_{q}[x,y]$, we write $f(x,y)=F(x)G(y)H(x,y)$ with the primitive kernel $H(x,y)$ of $f(x,y)$.
Then we say that $f$ is \emph{admissible} if it satisfies the following:
\begin{itemize}
    \item $f$ induces an undirected graph, that is, for each $u,v \in \F_q$, $f(u,v)$ is a square in $\F_q$ if and only if $f(v,u)$ is a square in $\F_q$;
    \item $H(x,y)$ is not a constant multiple of a square of a polynomial in $\F_q[x,y]$.
\end{itemize}
\end{definition}
Note that graphs of the form $X_{f,q}$ encompass a broad spectrum of well-known graphs as their special cases. 
For example, Paley graphs correspond to the case $f(x,y)=x-y$ for $q \equiv 1 \pmod 4$, and Paley sum graphs are realized when $f(x,y)=x+y$. \footnote{In some definitions of Paley sum graphs, two vertices are adjacent when their sum is a non-zero square. In some other references, loops are also allowed. These ambiguities do not affect the quasi-random property of Paley sum graphs.} 
Moreover, some of them are related to some well-known objects in number theory, for example, in the case of $f(x,y)=xy+1$. To be precise, a set $\{a_{1},a_{2},\ldots , a_{m}\}$ of distinct positive integers is called a \textit{Diophantine $m$-tuple} if the product of any two distinct elements in the set is one less than a perfect square. More generally, for an admissible $f$, a clique in the graph $X_{f,q}$ corresponds to an $f$-Diophantine set over $\F_q$. We refer to  \cref{subsec: Dio pules and graphs} for more background on Diophantine tuples, their generalizations, and related graphs.

In this paper, we explore the quasi-random property of graphs $X_{f,q}$ for admissible polynomials $f$. Roughly speaking, a combinatorial structure is called `quasi-random' if it is deterministic, but behaves similarly to a random structure. The systematic study of the quasi-randomness of graphs was initiated by Thomason \cite{T87b, T87a} and Chung, Graham, and Wilson \cite{CGW89}, independently. Thomason uses the terminology ``jumbled” as follows. Let $p,\alpha$ be positive real numbers with $p<1$. A graph $G=(V(G),E(G))=(V,E)$ is said to be \textit{$(p,\alpha)$-jumbled}  if every subset of vertices $S\subset V(G)$ satisfies:
\[\left| e(G[S]) - p\binom{|S|}{2} \right| \le \alpha |S|,\]
where $e(G[S])$ is the number of edges in the subgraph $G[S]$ induced by $S$ in $G$. Thus, the Erdős–Rényi random graphs $G(n,p)$ are almost surely $(p,\alpha)$-jumbled with $\alpha=O(\sqrt{np})$ \cite[Subsection 2.2]{KS06}. 
Chung, Graham, and Wilson \cite{CGW89} showed that several seemingly independent concepts of quasi-randomness of graphs are equivalent.
One of the concepts, denoted by property $P_4$ in \cite{CGW89}, is given as follows. 

\begin{definition}\label{def: quasirandom}
Let $\mathcal{G}$ be a family of graphs. Then $\mathcal{G}$ is a family of \textit{quasi-random graphs} if for each $G\in\mathcal{G}$ and each $S \subset V(G)$, we have
\begin{equation}\label{eq: quasi-random}
P_4: \quad e(S):=e(G[S])=\frac{1}{4}|S|^2+o(n^2)
\end{equation}
as $|V(G)|=n \to \infty$.   
\end{definition}

In the following discussion, we always assume $G$ is a graph with $n$ vertices. We remark that quasi-random graphs, which imitate random graphs with a more general edge density $p = p(n)$, have also been studied extensively; see, for example, \cite{CG02, T87a}. 
In \cite{CGW89}, the edge density $p$ is restricted to be near $1/2$, and we shall adopt this convention in this paper.
We also note that equation \cref{eq: quasi-random} is equivalent to 
\begin{equation}\label{eq: edge}
P_4': \quad e(S,T)=\frac{1}{2}|S||T|+o(n^2),  
\end{equation}
where $S,T\subset V$ and $e(S,T)$ is the number of edges of $G$ such that one end point is in $S$, and the other one is in $T$. 
Here, the edges whose endpoints are in $S \cap T$ are counted twice.
We refer to \cref{sec: equivalentdef} for a short proof as well as some equivalent definitions of quasi-random graphs introduced in \cite{CGW89}.

It is of special interest to construct quasi-random graphs for intriguing applications. 
Most known constructions of quasi-random graphs are Cayley graphs or Cayley sum graphs, which are described in \cref{subsec: constructions}. We will show that the graphs $X_{f,q}$ are quasi-random, and it is worthwhile to mention that most $X_{f,q}$ are neither Cayley graphs nor Cayley sum graphs.

For many applications (including \cref{cor: xfq} below for clique number and independence number), one needs to measure the quasi-randomness of graphs quantitatively. 
More precisely, we need to make the error term in equation \cref{eq: edge} effective. In the spirit of the expander mixing lemma, we introduce the following definition \footnote{A graph $G$ satisfying inequality~\eqref{eq: beta-quasi} is sometimes known as a $(\frac{1}{2}, cn^\theta)$-bi-jumbled graph in the literature; see for example \cite{KR07}. }.

\begin{definition} Let $0<\theta<1$ and let $\mathcal{G}$ be a family of graphs such that the number of vertices of graphs in $\mathcal{G}$ is unbounded. We say $\mathcal{G}$ is \emph{a family of quasi-random graphs with property $QR(\theta)$}, if there is a constant $c>0$, such that for all $G\in\mathcal{G}$, 
\begin{equation} \label{eq: beta-quasi}
\left|e(S,T)-\frac{|S||T|}{2}\right|\leq cn^\theta  \sqrt{|S||T|} 
\end{equation}
holds for any subsets $S,T \subset V(G)$, where $n=|V(G)|$.
\end{definition}

We recall that $\theta=1/2$ in the case of random graphs. Motivated by this fact, in \cref{sec: quasirandom}, we show that there is no family of quasi-random graphs with property $QR(\theta)$ with $\theta<1/2$ (\cref{prop: onehalf}). 
Thus, a family of quasi-random graphs having property $QR(1/2)$ is essentially the best possible, so it would be interesting to find families with property $QR(1/2)$. 
Next, we give a sufficient condition for checking the property $QR(\theta)$ by proving a quantitative version of the expander mixing lemma for quasi-random graphs in \cref{prop:expander mixing}, and \cref{cor: expander2}.
In \cref{sec: polynomial}, we prove the following result.

\begin{theorem}\label{thm: 3/4 intro}
Let $d \geq 1$. Let $\mathcal{X}_d$ be the family consisting of all graphs $X_{f,q}$, where $q$ is an odd prime power and $f \in \F_q[x,y]$ is an admissible polynomial with degree $d$. Then $\mathcal{X}_d$ is a family of quasi-random graphs with property $QR(3/4)$.
\end{theorem}

Furthermore, we also find that a subfamily of graphs in $\mathcal{X}_d$ has property $QR(1/2)$. This leads to the following construction of a new family of quasi-random graphs with property $QR(1/2)$.
 
\begin{theorem}\label{thm: construction intro}
Let $d \geq 1$. Let $\mathcal{H}_d$ be the family consisting of all graphs $X_{h,q}$ and $X_{\widetilde{h},q}$, where $q$ is an odd prime power, $h(x,y)=f(g(x),g(y))$ and $\widetilde{h}(x,y)=g(x)g(y)f(g(x),g(y))$, where $f \in \F_q[x,y]$ is an admissible polynomial with degree in $x,y$ both $1$, and $g \in \F_q[x]$ is a polynomial with degree $d$. Then $\mathcal{H}_d$ is a family of quasi-random graphs with property $QR(1/2)$.
\end{theorem}

In particular, Paley graphs $P_q=X_{x-y,q}$, Paley sum graphs $X_{x+y,q}$, and Diophantine graphs $D_q=X_{xy+1,q}$ (see \cref{def: dio graph}) form three families of quasi-random graphs with property $QR(1/2)$. 

\bigskip

Finally, we discuss lower bounds on clique number and independence number of the graphs $X_{f,q}$ by combining their quasi-randomness and recent breakthroughs in Ramsey theory.

For a general graph $X$ with order $n$, the classical 
Erd\H{o}s-Szekeres theorem \cite{ES35} implies that $\max\{\omega(X),\alpha(X)\}\geq (1-o(1))\log_4 n$. Using the recent breakthrough on diagonal Ramsey number by Campos, Griffiths, Morris, and Sahasrabudhe \cite{CGMS23}, one can deduce a better bound that $(1-o(1))\log_{3.993} n$. Recently, Gupta, Ndiaye, Norin, and Wei \cite{GNNW24} further optimized the proof techniques in \cite{CGMS23} and proved that $\max\{\omega(X),\alpha(X)\}\geq (1-o(1))\log_{3.8} n$.

It is known that for a quasi-random graph, one can use its quasi-random property and off-diagonal Ramsey bounds to get an improved lower bound on $\max\{\omega(X), \alpha(X)\}$; see for example Thomason \cite[Section 6]{T87a} and \cite[Section 4.2]{Thomason99}, as well as unpublished manuscripts \footnote{Private communication with J\'ozsef Solymosi.} of Alon \cite{Am}, Conlon \cite{Cm}, and Solymosi \cite{Sm}. In particular, from \cite{Sm, T87a}\footnote{Private communication with J\'ozsef Solymosi and Andrew Thomason.}, the best-known lower bound on the clique number of Paley graph $P_q$ is $\omega(P_q)\geq (1-o(1))\log_{3.008} q$, while the trivial lower bound is $(1-o(1))\log_4 q$ since Paley graphs are self-complementary. We refer to \cref{sec:Paley} for other relevant results on the clique number of Paley graphs.

In our case, since we have a quantitative measurement of the quasi-randomness of graphs $X_{f,q}$, we can follow the same idea to obtain lower bounds on $\max\{\omega(X_{f,q}),\alpha(X_{f,q})\}$. Moreover, we can take advantage of the best-known upper bound on Ramsey numbers, due to Gupta, Ndiaye, Norin, and Wei \cite{GNNW24} (building on \cite{CGMS23}). In particular, we show \cref{thm: 3/4 intro} and \cref{thm: construction intro} imply the following two corollaries. 

\begin{corollary}\label{cor: xfq}
Let $d$ be a positive integer. As $q \to \infty$, 
$$
\frac{\omega(X_{f,q})+\alpha(X_{f,q})}{2} \geq (1-o(1))\log_{3.501}q
$$
holds for all graphs $X_{f,q} \in \mathcal{X}_d$, and 
$$
\frac{\omega(X_{f,q})+\alpha(X_{f,q})}{2} \geq (1-o(1))\log_{2.936}q.
$$
holds for all graphs $X_{f,q} \in \mathcal{H}_d$. 
\end{corollary}
When $f$ is a homogeneous admissible polynomial with an odd degree, we strengthen \cref{cor: xfq} as follows. 

\begin{corollary}\label{thm: lower bound of homo}
Let $d$ be an odd positive integer. Let $f\in \F_q[x,y]$ be a homogeneous admissible polynomial with degree $d$. Then $\omega(X_{f,q}) \geq (1-o(1))\log_{3.501} q$ as $q \to \infty$; moreover, if $X_{f,q} \in \mathcal{H}_d$, then $\omega(X_{f,q}) \geq (1-o(1))\log_{2.936} q$. 
\end{corollary}

In particular, for Paley graphs, the above corollary implies that $\omega(P_q)\geq (1-o(1))\log_{2.936}q$.

\medskip

\textbf{Notation.} We follow standard notation in graph theory. Let us denote by $\omega(X)$ the clique number of a graph $X$, and $\alpha(X)$ the independence number of a graph $X$. We always assume $G$ is a graph with $n$ vertices. For a graph $G$, we use $\lambda_1(G),\lambda_2(G),\ldots, \lambda_n(G)$ to denote its eigenvalues such that $|\lambda_1(G)| \ge |\lambda_2(G)| \ge \cdots \ge |\lambda_n(G)|$. 
If there is no confusion, we denote $\lambda_i(G)$ by $\lambda_i$ for each $1\le i \le n$. For any vertex $v$ of $G$, we let $N(v)$ denote the neighborhood of $v$. We also follow standard asymptotic notations. 

\textbf{Organization of the paper.}
In \cref{sec: prelim1}, we provide more background about quasi-random graphs, Paley graphs, and Diophantine tuples. 
In \cref{sec: prelim}, we give some preliminary results for our proofs. 
In \cref{sec: quasirandom}, we introduce an expander mixing lemma for quasi-random graphs and prove a general proposition on clique number and independence number of quasi-random graphs in \cref{thm: 1.1}. 
In \cref{sec: polynomial}, we combine the tools to prove our new results: \cref{thm: 3/4 intro}, \cref{thm: construction intro}, \cref{cor: xfq}, and \cref{thm: lower bound of homo}.

\section{More backgrounds and motivations }\label{sec: prelim1}

\subsection{Known explicit constructions of quasi-random graphs}\label{subsec: constructions}
Classical examples of quasi-random graphs include cyclotomic graphs and cyclotomic sum graphs \cite{C91, CGW89, kp04,X12}. Cyclotomic sum graphs are Cayley graphs over the additive group of a finite field, with the connection set being the union of certain cosets of a fixed multiplicative subgroup. Cyclotomic sum graphs are similarly defined as Cayley sum graphs. In particular, these include Paley and Paley sum graphs. Another large family of quasi-random graphs is conference graphs (strongly regular graphs with specific parameters) \cite{BV22}, and many conference graphs come from finite geometry and design theory. In addition, Borb\'{e}ly and S\'{a}rk\"{o}zy \cite{BS19} constructed quasi-random graphs based on circulant matrices arising from pseudo-random binary sequences, and they also recovered Paley graphs as a special case \cite[Theorem 5.1]{BS19}. We refer to \cite{C91, CGW89, T87b, T16} for more explicit constructions of quasi-random graphs. It is interesting to note that the most explicit constructions of quasi-random graphs are either Cayley graphs or Cayley sum graphs, including the aforementioned explicit constructions. In particular, Cayley graphs and Cayley sum graphs are regular, and their eigenvalues can be explicitly computed via character sums; thus, one can use the expander mixing lemma to verify their quasi-randomness readily. 

\subsection{Clique number of Paley graphs}\label{sec:Paley}
Estimating the clique number of Paley graphs is an important open problem, due to the rich connections between Paley graphs and other problems in additive combinatorics, analytic number theory, and algebraic graph theory \cite{CL07}. We refer to \cite{Y25a} for a discussion on the state of the art of the bounds on the clique number of Paley graphs. The trivial upper bound on $\omega(P_q)$ is $\sqrt{q}$ \cite[Lemma 1.2]{Yip1}, and it is tight when $q$ is a square. We assume that $q$ is a non-square in the following discussion. 

When $q$ is a prime, it is widely believed that $\omega(P_q)=O_{\epsilon} (q^{\epsilon})$ for each $\epsilon>0$ (in fact, this is a special case of the more general Paley graph conjecture). However, until very recently, the trivial upper bound $\sqrt{q}$ on $\omega(P_q)$ has been improved only to $\sqrt{q}-1$. Recently, Hanson and Petridis \cite{HP} and Yip \cite{Yip1} improved the $\sqrt{q}$ bound to $(1+o(1))\sqrt{q/2}$ using Stepanov's method.

As for the lower bound, Graham and Ringrose \cite{GR} showed that there is a constant $C>0$, such that $\omega(P_q)\geq C\log q \log \log \log q$ holds for infinitely many primes $q \equiv 1 \pmod 4$. Moreover, under the generalized Riemann hypothesis (GRH), Montgomery \cite{HLM} further improved the $C\log q \log \log \log q$ bound to $C\log q \log \log q$.  

\subsection{More on Diophantine tuples and Diophantine graphs}\label{subsec: Dio pules and graphs}

Diophantine $m$-tuples and Paley graphs are important and well-studied. They were treated completely independently until recently G\"{u}lo\u{g}lu and Murty pointed out the connection between the Paley graph conjecture and the size of Diophantine $m$-tuples in \cite{GM20}. 
This interesting connection has been further explored in \cite{DKM22, KYY, KYY24a, Y25}. 
Inspired by such a connection, we define a new graph associated with Diophantine $m$-tuples over $\mathbb{F}_{q}$, which encodes their properties. 

\begin{definition}\label{def: dio graph}
Let $q$ be an odd prime power. The \textit{Diophantine graph} $D_{q}$ is the graph whose vertex set is $\mathbb{F}_{q}$, and two vertices $x$ and $y$ are adjacent if and only if $xy+1$ is a square in $\mathbb{F}_{q}$. 
\end{definition}
Based on the definition of Diophantine $m$-tuples over $\mathbb{F}_{q}$, we do not consider loops in $D_{q}$. Also, we note that a Diophantine tuple over $\F_q$ is exactly a clique in $D_q$.
Perhaps the most important problem in the study of Diophantine tuples over $\F_q$ is to determine the size of the largest Diophantine tuple over $\F_q$, equivalently, to find the clique number $\omega(D_q)$ of $D_q$. 
Although the above definition of Diophantine graphs appears to be new, tools from extremal graph theory and Ramsey theory have been frequently applied to study Diophantine tuples and their generalizations; see, for example, recent papers \cite{BHP25, Y25small}. In particular, Gyarmati \cite[Theorem 4]{G01} has already studied the symmetry of Diophantine graphs and proved $\omega(D_q)\geq (1-o(1))\log_{729} q$ using three-color Ramsey numbers.
We believe that obtaining a precise estimate on $\omega(D_q)$ is of the same difficulty as estimating $\omega(P_q)$, or perhaps even slightly harder.
 
Similar to the case of Paley graphs, one can verify that Diophantine graphs form a family of quasi-random graphs with property $QR(1/2)$ using \cref{thm: construction intro}.
Then \cref{thm: 1.1} implies the following corollary.

\begin{corollary}\label{cor: max dio}
When $q \to \infty$, we have
$$
\max \{\omega(D_q),\alpha(D_q)\} \geq (1-o(1))\log_{2.936}q.
$$
\end{corollary}

One may ask if we have $\omega(D_q) \ge (1-o(1))\log_{2.936}q$. 
This would improve the lower bound of the largest size of Diophantine tuples over $\F_q$.
Dujella and Kazalicki \cite[Theorem 17]{21DK} showed that the lower bound of the largest size of Diophantine tuples is $(1-o(1))\log_4 q$. 
Very recently, Kim, Yip, and Yoo \cite[Corollary 3.2, Theorem 3.5]{KYY24a} improved their bound to $(\frac{p}{p-1}-o(1))\log_4 q$ as $q=p^m \to \infty$, for a fixed $p$ such that $q \not \equiv 3 \pmod 8$. This is the best-known lower bound.

Using \cref{cor: xfq}, we can also obtain an lower bound on the average clique number of $X_{f,q}$ when $f=axy+b$ with $a,b \in \F_q^*$.
When $a=1,b=1$, this yields the case of the Diophantine graph $D_q$. Moreover, when $a=1$ and $b$ is nonzero, this corresponds to generalized Diophantine tuples over $\F_q$.
\begin{corollary}\label{corollary: average clique}
Let $q$ be an odd prime power. If we consider the family of graphs $X_{f,q}$ arising from the collection of polynomials $f(x,y)=axy+b$ with $a,b \in \F_q^*$, the average clique number of these graphs is at least $(1-o(1))\log_{2.936} q$ as $q \to \infty$.
\end{corollary}

\begin{proof}
Let $r$ be a fixed non-square in $\F_q$. We have $\omega(X_{rf, q})\geq \alpha(X_{f,q})$ since an independent set in $X_{f,q}$ is a clique in $X_{rf, q}$. Also note that the map $f \mapsto rf$ is a bijection on $\{axy+b: a,b \in \F_q^*\}$. Thus, by \cref{thm: construction intro} and \cref{cor: xfq}, 
\begin{align*}
2\sum_{a,b \in \F_q^*} \omega(X_{f,q})
&=\sum_{a,b \in \F_q^*} \omega(X_{f,q})+\sum_{a,b \in \F_q^*} \omega(X_{rf,q})=\sum_{a,b \in \F_q^*} \big(\omega(X_{f,q})+\omega(X_{rf,q})\big) \\
&\geq \sum_{a,b \in \F_q^*} \big(\omega(X_{f,q})+\alpha(X_{f,q})\big) \geq 2 \sum_{a,b \in \F_q^*} (1-o(1)) \log_{2.936}q.
\end{align*}
It follows that the average clique number is at least $(1-o(1))\log_{2.936} q$.
\end{proof}

Furthermore, it is possible to explore more general cases.
The definition of $f$-Diophantine sets was formally introduced by B\'{e}rczes, Dujella, Hajdu, Tengely \cite{BDHT16} for a polynomial $f \in \Z[x,y]$. Given a polynomial $f(x,y) \in \Z[x,y]$, a set $A\subset\mathbb{Z}$ is an \emph{$f$-Diophantine set} if $f(x,y)$ is a perfect square for all $x,y \in A$ with $x \neq y$. $f$-Diophantine sets are related to many famous problems in number theory \cite{BDHT16}. In particular, the case $f(x,y)=xy+1$ corresponds to the well-studied Diophantine tuples. Motivated by their definition, the exact analogue of $f$-Diophantine sets over finite fields was introduced in \cite{YY25}. If $q$ is an odd prime power and $f=f(x,y) \in \F_q[x,y]$ is a polynomial, we say $A\subset \mathbb{F}_q$ is an \emph{$f$-Diophantine set over $\F_q$} if $f(x,y)$ is a square in $\F_q$ for all $x,y \in A$ with $x \neq y$. Note that the graph $X_{f,q}$ is naturally induced by $f$-Diophantine sets over $\F_q$. Similar to the study of classical Diophantine tuples, it is of special interest to give bounds on the maximum size of $f$-Diophantine sets over $\F_q$, or equivalently, using the graph theory language, estimate the clique number of $X_{f,q}$.

The quasi-randomness of Diophantine graphs has many applications in the study of Diophantine tuples over finite fields, and so do their generalizations. As an illustration, if $m$ is fixed, and $q \to \infty$, Dujella and Kazalicki \cite{21DK} showed that the number of Diophantine $m$-tuples over $\F_q$ is $q^m/(2^{\binom{m}{2}}m!)+o(q^m)$. Since Diophantine graphs are quasi-random, this also follows from the property $P_1(m)$ from \cite{CGW89}. In view of \cref{thm: 3/4 intro}, if $d,m$ are fixed and $q \to \infty$, we have the same asymptotic for the number of $f$-Diophantine sets over $\F_q$ with size $m$, provided $X_{f,q} \in \mathcal{X}_d$. On the other hand, the techniques used in \cite{21DK} do not seem to extend to this general setting.

\section{Preliminaries}\label{sec: prelim}

Throughout the section, let $f \in \F_q[x,y]$ be a polynomial with degree $d \geq 1$. Recall that Gyarmati and S\'{a}rk\"{o}zy \cite[Proposition 1]{GS08} showed that one can always write $f(x,y)=F(x)G(y)H(x,y)$, where $H(x,y)$ is primitive in both $x$ and $y$. Following their notations, we call $H(x,y)$ the primitive kernel of $f(x,y)$, and we note that $H(x,y)$ is unique up to a constant factor. We prove some properties of admissible polynomials and discuss their implications for graphs $X_{f,q}$. These preliminary results will be useful for the proof of our main results.

\subsection{A necessary condition for  $X_{f,q}$ to be Paley-like}
In this subsection, we motivate the assumptions in the definitions of admissible polynomials: we show that if $X_{f,q}$ behaves like a Paley graph, then $f$ is necessarily admissible. More precisely, if the primitive kernel $H(x,y)$ is a constant multiple of a square of a polynomial, then the graph $X_{f,q}$ has a simple structure. 

\begin{lemma}\label{lem:NIZ}
For each $a \in \overline{\F_q}$, $H(a,y)$ is not identically zero.
\end{lemma}
\begin{proof}
Suppose the polynomial $H(a, y)$ is identically zero. Then $H(x,y)=H(x,y)-H(a,y)$ has the factor $x-a$. Let $k(x)$ be the minimal polynomial of $a$ over $\F_q$; then $k(x)$ divides $H(x,y)$,contradicting the definition of $H$.    
\end{proof}

\begin{lemma}\label{lem:natural}
Let $f \in \F_q[x,y]$ be of degree $d$, such that $X_{f,q}$ is an undirected graph. Let $H$ be the primitive kernel of $f(x,y)$. If $H(x,y)$ is a constant multiple of a square of a polynomial, then by removing at most $2d$ vertices and $dq$ edges from the graph $X_{f,q}$, it becomes either a complete graph, an empty graph, a complete bipartite graph, or the vertex-disjoint union of two complete subgraphs.  
\end{lemma} 

\begin{proof}
We can write $f(x,y)=F(x)G(y)H(x,y)$. By scaling $F(x)$ properly, we may assume that $H(x,y)$ is a square of a polynomial. 

Let $E=\{u \in \F_q: F(u)=0\} \cup \{v \in \F_q: G(v)=0\}$; it is clear that $|E|\leq 2d$. Let $Y$ be the subgraph of $X_{f,q}$, with vertex set $\F_q \setminus E$, such that two distinct vertices $u,v$ are adjacent if and only if $f(u,v)$ is a non-zero square in $\F_q$. Note that for each $u \in \F_q \setminus E$, the polynomial $f(u,y)$ is not identically zero by \cref{lem:NIZ} and the assumption $F(u) \neq 0$, thus there are at most $d$ many $u \in \F_q$ such that $f(u,v)=0$. This shows that the graph $Y$ can be obtained from $X_{f,q}$ by removing at most $2d$ vertices and $dq$ edges.

Next, we show that $Y$ has the desired property. Let $A_1=\{u \in \F_q \setminus E: F(u) \text{ is a square in } \F_q\}$, and let $A_2=\F_q \setminus E \setminus A_1$. Similarly, define $B_1=\{v \in \F_q \setminus E: G(v) \text{ is a square in } \F_q\}$, and let $B_2=\F_q \setminus E \setminus B_1$. Consider the following 4 sets $A_1 \cap B_1$, $A_1 \cap B_2$, $A_2 \cap B_1$, $A_2 \cap B_2$, which form a partition of $\F_q \setminus E$. If both $A_1 \cap B_1$ and $A_1 \cap B_2$ are nonempty, then for $u \in A_1 \cap B_1$ and $v \in A_1 \cap B_2$, $f(u,v)=F(u)G(v)H(u,v)$ is a non-square, while $f(v,u)=F(v)G(u)H(v,u)$ is a square, violating the assumption that $f$ induces an undirected graph. Thus, either $A_1 \cap B_1$ or $A_1 \cap B_2$ is empty. Similarly, we can show that either $A_1 \cap B_1$ or $A_2 \cap B_1$ is empty, either $A_1 \cap B_2$ or $A_2 \cap B_2$ is empty, and either $A_2 \cap B_1$ or $A_2 \cap B_2$ is empty. Therefore, at most 2 of the sets of $A_1 \cap B_1$, $A_1 \cap B_2$, $A_2 \cap B_1$, $A_2 \cap B_2$ are nonempty. If there are 2 of the sets of $A_1 \cap B_1$, $A_1 \cap B_2$, $A_2 \cap B_1$, $A_2 \cap B_2$ are nonempty, then they must be either $A_1 \cap B_1$ and $A_2 \cap B_2$, or $A_1 \cap B_2$ and $A_2 \cap B_1$-- in both cases, the graph $Y$ is either the vertex-disjoint union of two complete subgraphs, or a complete bipartite graph. If only one set in $A_1 \cap B_1$, $A_1 \cap B_2$, $A_2 \cap B_1$, $A_2 \cap B_2$ is nonempty, then the graph $Y$ is either a complete graph or an empty graph.    
\end{proof}

\subsection{Basic properties of admissible polynomials}

In this subsection, we additionally assume that the primitive kernel $H$ is not a multiple of a square of a polynomial.

\begin{lemma}\label{lem:Hsquare}
The number of $u \in \F_q$ such that $H(x,u)$ is a constant multiple of a square of a polynomial is at most $d^2+d$.    
\end{lemma}
\begin{proof}
This is implicit in the proof of \cite[Theorem 5]{GS08}. In \cite[Lemma 4]{GS08}, Gyarmati and S\'{a}rk\"{o}zy defined the set $Y_1=\{u \in \F_q: H(x,u)=ch^2(x) \text{ for some } c \in \F_q \text{ and } h(x) \in \F_q[x]\}$ and proved that $|Y_1|\leq d^2+d$.
\end{proof}

Next, we deduce two corollaries.

\begin{corollary}\label{cor:badu}
The number of $u \in \F_q$ such that $f(x,u)$ is a constant multiple of a square of a polynomial is $O(d^2)$. 
\end{corollary}

\begin{proof}
If $f(x,u)$ is a constant multiple of a square of a polynomial, then at least one of the following happens:
\begin{itemize}
    \item $G(u)=0$.
    \item $H(x,u)$ is a constant multiple of a square of a polynomial.
    \item $\gcd(F(x),H(x,u))$ has degree at least $1$. 
\end{itemize}
The number of $u$ in the first situation is at most $d$, since the degree of $G$ is at most $d$. By \cref{lem:Hsquare}, the number of $u$ in the second situation is at most $O(d^2)$. In the third case, $F(x)$ and 
$H(x,u)$ have a common root $c$ in $\overline{\F_q}$. The common root $c$ has at most $d$ choices. Since the polynomial $H(x,c)$ is not identically zero by \cref{lem:NIZ}, the number of $u$ such that $H(x,u)=0$ is at most $d$. Thus, the number of $u$ in the third family is at most $d^2$. 
\end{proof}

We can prove the following corollary similarly.

\begin{corollary}\label{cor:badpairs}
Let $C \subset \F_q$. The number of pairs $(u,v) \in C \times C$ such that $f(x,u)f(x,v)$ is a constant multiple of a square of a polynomial is at most $(d^2+d)^2+(d^2+2d)|C|$.    
\end{corollary}

\begin{proof}
Let $(u,v) \in C \times C$. If $f(x,u)f(x,v)=F(x)^2G(u)G(v)H(x,u)H(x,v)$ is a constant multiple of a square of a polynomial, then one of the following holds:
\begin{itemize}
    \item $G(u)G(v)=0$. 
    \item $H(x,u)$ and $H(x,v)$ are both a constant multiple of a square of a polynomial.
    \item $H(x,u)$ is not a constant multiple of a square of a polynomial, and $\gcd(H(x,u), H(x,v))$ has a degree at least $1$. 
\end{itemize}

Clearly, the number of pairs $(u,v)$ in the first family is at most $2d|C|$. By \cref{lem:Hsquare}, the second family has size at most $(d^2+d)^2$. 

For the third family, $H(x,u), H(x,v)$ have a common root $c$ in $\overline{\F_q}$. Fix $u$, the common root $c$ has at most $d$ choices since $H(x,u)$ is not identically zero. We have $H(c,v)=0$. Since the polynomial $H(c, y)$ is not identically zero by \cref{lem:NIZ}, the number of $v$ such that $H(c,v)=0$ is at most $d$. Thus, the number of pairs $(u,v)$ in the third family is at most $d^2|C|$. 

We conclude that the number of pairs $(u,v) \in C \times C$ such that $f(x,u)f(x,v)$ is a constant multiple of a square of a polynomial is at most $(d^2+d)^2+(d^2+2d)|C|$.    
\end{proof}

\subsection{Upper bound on the clique number of $X_{f,q}$}\label{sec:ub}
In this subsection, we derive an upper bound on $\omega(X_{f,q})$. Note that for a family $\mathcal{G}$ of quasi-random graphs with property $QR(\theta)$ and a graph $G\in\mathcal{G}$ with $n$ vertices, we have $\omega(G)=O(n^\theta)$. Indeed, this follows from a simple observation: if $S$ is a maximum clique in $G$, then we have $e(S)=|S|(|S|-1)/2$; on the other hand, we know that $e(S)=|S|^2/4+O(n^\theta|S|)$. \cref{thm: 3/4 intro} thus implies that $\omega(X_{f,q})=O_d(q^{3/4})$, whenever $f$ is an admissible polynomial with degree $d$. However, using tools from character sums, we show the better bound $\omega(X_{f,q})=O_d(\sqrt{q})$ in \cref{cor:ub} below.

We first recall Weil's bound for complete character sums; see, for example, \cite[Theorem 5.41]{LN97}.

\begin{lemma}(Weil's bound)\label{Weil}
Let $\chi$ be a multiplicative character of $\F_q$ of order $k>1$, and let $g \in \F_q[x]$ be a monic polynomial of positive degree that is not an $k$-th power of a polynomial. 
Let $n$ be the number of distinct roots of $g$ in its
splitting field over $\F_q$. Then for any $a \in \F_q$,
$$\bigg |\sum_{x\in\mathbb{F}_q}\chi\big(ag(x)\big) \bigg|\le(n-1)\sqrt q \,.$$
\end{lemma}

Next, we use Weil's bound to deduce an upper bound estimate on an incomplete $2$-dimensional character sum.

\begin{proposition}\label{prop:charsum}
Let $C \subset \F_q$. Let $f \in \F_q[x,y]$ be of degree $d$, such that its primitive kernel $H$ is not a multiple of a square of a polynomial. Let $\chi$ be the quadratic character. Then
$$
\sum_{a,b \in C} \chi(f(a,b))=O_d(|C|^{3/2}q^{1/4}+|C|q^{1/2}).
$$        
\end{proposition}
\begin{proof}
By the Cauchy-Schwarz inequality,
$$
\bigg |\sum_{a,b \in C} \chi(f(a,b))\bigg |^2 \leq |C| \bigg(\sum_{a \in C} \bigg|\sum_{b \in C}\chi(f(a,b))\bigg|^2\bigg) \leq |C| \bigg(\sum_{x \in \F_q} \bigg|\sum_{b \in C}\chi(f(x,b))\bigg|^2\bigg).
$$
Note that
$$
\sum_{x \in \F_q} \bigg |\sum_{b \in C}\chi(f(x,b))\bigg |^2= \sum_{u, v \in C}\bigg(\sum_{x \in \F_q} \chi(f(x,u)f(x,v))\bigg).
$$
If $f(x,u)f(x,v)$ is not a constant multiple of a square of a polynomial, then Weil's bound implies the inner sum is bounded by $d\sqrt{q}.$ On the other hand, \cref{cor:badpairs} implies that the number of pairs $(u,v) \in C \times C$ such that $f(x,u)f(x,v)$ is a constant multiple of a square of a polynomial is at most $(d^2+d)^2+(d^2+2d)|C|$. Therefore,
$$
\sum_{u, v \in C} \bigg |\sum_{x \in \F_q} \chi(f(x,u)f(x,v))\bigg|\leq |C|^2d\sqrt{q}+((d^2+d)^2+(d^2+2d)|C|)q  =O_d(|C|^2\sqrt{q}+|C|q).
$$
We conclude that
$$
\bigg |\sum_{a,b \in C} \chi(f(a,b))\bigg |^2 =O_d(|C|^3q^{1/2}+|C|^2q),
$$
as required.
\end{proof}

\begin{corollary}\label{cor:ub}
Let $f \in \F_q[x,y]$ be admissible with degree $d$. Then $\omega(X_{f,q})=O_d(\sqrt{q})$.  
\end{corollary}
\begin{proof}
Let $\chi$ be the quadratic character of $\F_q$. Let $C$ be a maximum clique of $X_{f,q}$. For each $a,b \in C$ with $a \neq b$, we have $\chi(f(a,b))=1$ unless $f(a,b)=0$. Note that if $f(x,b)$ is not identically zero, then the number of $a \in C$ such that $f(a,b)=0$ is at most $d$. On the other hand, by \cref{cor:badu}, the number of $b \in C$ such that $f(x,b)$ is the zero polynomial is $O(d^2)$. It follows that
$$
\sum_{a,b \in C} \chi(f(a,b)) \geq \sum_{a,b \in C, a \neq b} \chi(f(a,b)) -|C| \geq |C|^2-2|C|-d|C|-O(d^2)|C|=|C|(|C|-O(d^2)).
$$
Combining the above lower bound estimate on $\sum_{a,b \in C} \chi(f(a,b))$ with the upper bound stated in \cref{prop:charsum}, we conclude that $\omega(X_{f,q})=|C|=O_d(\sqrt{q})$.
\end{proof}

\begin{remark}
\cref{cor:ub} is best possible if $q$ is a square and $f \in \F_{\sqrt{q}}[x,y]$. Indeed, if that is the case, then the subfield 
$\F_{\sqrt{q}}$ forms a clique in $X_{f,q}$ since all elements in $\F_{\sqrt{q}}$ are squares in $\F_q$. In particular, if $q$ is a square, then $\omega(P_q)=\sqrt{q}$.
\end{remark}

\section{quasi-random graphs} \label{sec: quasirandom}

\subsection{Equivalent definitions of quasi-random graphs}\label{sec: equivalentdef}
Let $\mathcal{G}$ be a family of graphs. In \cite{CGW89}, Chung, Graham, and Wilson showed that the following statement, which they call $P_3$, is equivalent to $P_4$ in \cref{def: quasirandom}:
If for each $G \in \mathcal{G}$, we have 
$$
P_3: \quad e(G) \ge (1+o(1))\frac{n^2}{4}, \quad \lambda_1=(1+o(1))\frac{n}{2}, \quad\lambda_2=o(n)
$$
as $|V(G)|=n\to \infty$. 
We also note that $P_4$ is equivalent to $P_4'$ in \cref{sec: intro}, and we now give a short proof of this argument.
It is obvious that $P_4'$ implies \cref{def: quasirandom}. To see the converse, note that \cref{def: quasirandom} implies that if $S,T$ are disjoint, then $e(S,T)=e(S \cup T)-e(S)-e(T)=\frac{1}{2}|S||T|+o(n^2)$. If $S,T$ are not disjoint, let $W=S \cap T$. 
Hence, we have $$e(S,T)=e(S\setminus W, T \setminus W)+2e(W)+e(S\setminus W, W)+e(T \setminus W,W)=\frac{1}{2}|S||T|+o(n^2).$$
Lastly, if $\mathcal{G}$ is a family of quasi-random graphs, then the following $P_0'$ (proved in \cite{CGW89}) holds:
\[
P_0': \text{ For each $G\in \mathcal{G}$, all but }o(|V(G)|) \text{ vertices of } G \text{ have degree }(1+o(1))\frac{|V(G)|}{2}.
\]

\subsection{No family of quasi-random graphs has property $QR(\theta)$ with $\theta<1/2$} \label{subsec: no less than 1/2} 

In this subsection, we prove the following proposition.
\begin{proposition}\label{prop: onehalf}
There is no family of quasi-random graphs with property $QR(\theta)$, where $\theta<1/2$. 
\end{proposition}

\begin{proof}
Let $\mathcal{G}$ be a family of quasi-random graphs with property $QR(\theta)$ such that $\theta<1/2$. Let $G\in\mathcal{G}$ be a graph with $n$ vertices. Using $P_0'$, we can take an induced subgraph $G'$ with $(1-o(1))n$ vertices such that all vertices in $G'$ have degree $(1+o(1))\frac{n}{2}$. Therefore, $\lambda_1(G')=(1+o(1))\frac{n}{2}$ and
$$\lambda_1(G')^2+\lambda_2(G')^2+\cdots +\lambda_{|V(G')|}(G')^2= \textup{tr}((G')^{2})= 2e(G')=\frac{1}{2}n^2+o(n^2).$$ 
It follows that 
$$
\lambda_2(G')^2 \geq \frac{\frac{1}{2}n^2+o(n^2)-(1+o(1))(\frac{n}{2})^2}{n-1}=\frac{1}{4}n+o(n).
$$
By Cauchy's interlacing theorem, we have $|\lambda_2(G)|\geq (\frac{1}{2}+o(1))\sqrt{n}$. 
Besides, as $G\in\mathcal{G}$, we have 
\begin{equation}\label{eq:est}
e(S,T)=\frac{|S||T|}{2}+O(n^\theta  \sqrt{|S||T|}),
\end{equation}
for any subsets $S,T \subset V(G)$. In particular, we have $e(G)=\frac{n^2}{4}+O(n^{1+\theta})$. Let $\Delta$ be the maximum degree of $G$; then trivially $\Delta \leq n$. Therefore, $\Delta/|\lambda_2(G)|=O(\sqrt{n})$. 

On the other hand, using the partial converse to the expander mixing lemma \cite[Corollary 4]{L15} due to Lev, there are non-empty subsets $S,T \subset V(G)$, such that
\begin{equation}\label{eq:cem}
\bigg |e(S,T)- \frac{2e(G)}{n^2}|S||T|\bigg| \geq \frac{|\lambda_2(G)|}{32\sqrt{2} (\log(2\Delta/|\lambda_2(G)|)+4}\sqrt{|S||T|}.
\end{equation}
Note that
\begin{equation}\label{eq:eg}
\frac{2e(G)}{n^2}|S||T|=\frac{1}{2}|S||T|+O(n^{\theta-1}|S||T|)=\frac{1}{2}|S||T|+O(n^{\theta}\sqrt{|S||T|}).
\end{equation}
It follows from equations~\cref{eq:est}, \cref{eq:cem}, and \cref{eq:eg} that
$$
O(n^\theta \sqrt{|S||T|})=\bigg |e(S,T)- \frac{1}{2}|S||T|\bigg|\geq \frac{|\lambda_2(G)|}{32\sqrt{2} (\log (2\Delta/|\lambda_2(G)|)+4}\sqrt{|S||T|}-O(n^{\theta}\sqrt{|S||T|}),
$$
and thus
$$
\frac{|\lambda_2(G)|}{32\sqrt{2} (\log (2\Delta/|\lambda_2(G)|)+4}=O(n^\theta),
$$
which contradicts the assumption that $\theta<1/2$ since the left-hand side is $\gg \sqrt{n}/\log n$.
\end{proof}

\subsection{An expander mixing lemma for quasi-random graphs} 
In general, it is hard to verify if a given family of graphs has property $QR(\theta)$ for a given $\theta>0$ by checking all possible subsets of vertices of graphs in the family. In this subsection, we provide a sufficient condition in terms of a quantitative version of $P_3$ for checking the property $QR(\theta)$. This sufficient condition can also be viewed as an expander mixing lemma for quasi-random graphs.
Our proof is a refined version of the proof that $P_3$ implies $P_4$ by Chung, Graham, and Wilson \cite[Theorem 1]{CGW89}.

\begin{lemma}\label{prop:expander mixing}
Let $\alpha<3$ and $\beta, \gamma<1$. Let $\mathcal{G}$ be a family of graphs such that the number of vertices of graphs in $\mathcal{G}$ is unbounded.
\begin{enumerate}
    \item[\textup{(1)}] Assume that all graphs $G$ in the family $\mathcal{G}$ satisfy the following properties:
$$
\sum_{v \in V(G)} \Big(\deg(v)-\frac{n}{2}\Big)^2=O(n^\alpha), \quad \lambda_1(G)=\frac{n}{2}+O(n^\beta), \quad\lambda_2(G)=O(n^\gamma),
$$    
where $n=|V(G)|$.
Then, $\mathcal{G}$ is a family of quasi-random graphs with property $QR(\theta)$,
where 
$$\theta=\max((1+\alpha)/4, \beta, \gamma).$$
\item[\textup{(2)}] 
Moreover, if each graph $G \in \mathcal{G}$ has the additional property that all but $O(1)$ vertices in $G$ have degree $n/2+O(n^\eta)$ with $1/2\leq \eta<1$, then $\mathcal{G}$ is a family of quasi-random graphs with property $QR(\theta)$ with $\theta=\max(\eta, \beta, \gamma).$ 
\end{enumerate}
\end{lemma}

Suppose that $\mathcal{G}$ is a family of graphs with all conditions assumed in \cref{prop:expander mixing} and let $G \in \mathcal{G}$ with $V=V(G)=\{v_1,v_2,\ldots,v_n\}$ and adjacency matrix $A$. Let $\{\bar{e}_i \mid i=1,2,\ldots,n\}$ denote a set of orthonormal eigenvectors corresponding to the eigenvalues $\{\lambda_i \mid i=1,2,\ldots,n\}$ of $A$, and set $\bar{u}=\frac{1}{\sqrt{n}}(1,1,\ldots,1)^t=\sum_i b_i \bar{e}_i$. The following claim is crucial in our proof.

\begin{claim}
\label{claim1}
There exists a vector $\bar{z_1}$ such that
$$\bar{u}= (1+O(||\bar{z_1}||))\bar{e_1} + \bar{z_1} \quad \text{with } \quad ||\bar{z_1}||=O(n^{\frac{\alpha-3}{4}}).
$$
Moreover, under the assumption of $(2)$, we can choose $\bar{z_1}$ such that $||\bar{z_1}||=O(n^{\eta-1})$.
\end{claim}

\begin{proof}
For any $\epsilon>0$, we have the inequality 
\[\#\{v \in V(G): \left| \deg(v) - \frac{n}{2}\right| > \epsilon n \} \le \frac{\sum_v (\deg(v)-\frac{n}{2})^2}{\epsilon^2 n^2}= \frac{O(n^{\alpha})}{\epsilon^2n^2}.\]
Thus, for any $\epsilon>0$, all but $\frac{1}{\epsilon^2 n^2}O(n^{\alpha})$ vertices of $G$ have degree in $[\frac{n}{2}-\epsilon n, \frac{n}{2}+\epsilon n]$. 
Recall that $\bar{u}=\frac{1}{\sqrt{n}}(1,1,\ldots,1)^t=\sum_i b_i \bar{e_i}$. Then $A\bar{u}=\sum_i b_i \lambda_i \bar{e_i}$. On the other hand, the $j$th component of the vector $A\bar{u}$ is $\deg(v_j)/\sqrt{n}$, where $v_j$ is the $j$th vertex of $G$. This implies that we can write
\begin{equation}\label{eq: Aubar}
A\bar{u}=\frac{1}{2}n\bar{u} + \bar{w},
\end{equation}
where all but $\frac{1}{\epsilon^2 n^2}O(n^{\alpha})$ components of $\bar{w}$ have absolute value at most $\epsilon \sqrt{n}$. 
Also, note each component of $\bar{w}$ is trivially bounded above by $\frac{1}{2}\sqrt{n}$ since each component of $A\bar{u}$ is trivially in $[0,\sqrt{n}]$.
It then follows that 
$$
||\bar{w}||^2\leq \frac{O(n^{\alpha})}{\epsilon^2n^2}\cdot n+n \cdot \epsilon^2n.
$$
Now we optimize the function of $\epsilon$ on the right-hand side of the above inequality by setting $\epsilon=n^{\frac{\alpha-3}{4}}$, and thus we have 
\begin{equation}\label{ineq: wbar}
||\bar{w}||\leq O(n^{\frac{\alpha+1}{4}}).
\end{equation}
Note that by equation \cref{eq: Aubar},
\[\sum_{i =1}^n\left(\lambda_i -\frac{n}{2}\right)b_i\bar{e}_i=\bar{w}.\]
This, together with inequality \cref{ineq: wbar}, implies that
\[\left( \sum_{i=1}^n\left(\lambda_i -\frac{n}{2}\right)^2 b_i^2 \right)^{\frac{1}{2}} = ||\bar{w}|| \le O(n^{\frac{\alpha+1}{4}}).\]
Since 
\[\left( \sum_{i \ne 1}\left(O(n^{\gamma}) -\frac{n}{2}\right)^2 b_i^2 \right)^{\frac{1}{2}}=\left|O(n^{\gamma})-\frac{n}{2}\right|\left( \sum_{i \ne 1} b_i^2 \right)^{\frac{1}{2}} \le ||\bar{w}||\le O(n^{\frac{\alpha+1}{4}}),\]
we obtain
\[\left( \sum_{i \ne 1} b_i^2 \right)^{\frac{1}{2}} \le \frac{O(n^{\frac{\alpha+1}{4}})}{\left|O(n^{\gamma})-\frac{n}{2}\right|} = O(n^{\frac{\alpha-3}{4}}).\]
Thus we have a vector $z_1$ such that $\bar{u}=b_1\bar{e}_1+\bar{z}_1$ with $||\bar{z_1}|| \le O(n^{\frac{\alpha-3}{4}})$.
In addition, by applying the triangle inequality to $||\bar{u}||=||b_1\bar{e}_1+\bar{z}_1||=1$, we have $|b_1|=1+O(||\bar{z_1}||)$.
By the Perron-Frobenius theorem,  $b_1=1+O(||\bar{z_1}||)$, and thus $\bar{u}= (1+O(||\bar{z_1}||))\bar{e_1} + \bar{z_1}$, as desired.    

Under the assumption of $(2)$, we have $||\bar{w}||^2 \leq O(n^{2\eta})+O(n)=O(n^{2\eta})$. Thus $||\bar{w}||=O(n^{\eta})$, and a similar argument shows that $||\bar{z_1}||=O(n^{\eta-1})$.
\end{proof}

Now we are ready to prove \cref{prop:expander mixing}.
\begin{proof}[Proof of \textup{\cref{prop:expander mixing}}]
Let $S,T \subset V$ and let $\bar{s}=(s_1,\ldots,s_n)$ be the characteristic vector of $S$. 
In other words, we have $s_i=1$ if $v_i\in S$, and $s_i=0$ otherwise.
Let $\bar{s}'=\bar{s}-\left \langle \bar{s},\bar{e}_1 \right \rangle \bar{e_1}$. By \cref{claim1}, we have
\begin{align}\label{eq: se}
\left \langle \bar{s},\bar{e}_1 \right \rangle 
=\frac{|S|}{\sqrt{n}}-\left \langle \bar{s},\bar{z}_1 + O(||\bar{z}_1||)\bar{e}_1\right \rangle 
=\frac{|S|}{\sqrt{n}}+ O(||\bar{s}||||\bar{z_1}||) 
& =\frac{|S|}{\sqrt{n}}+ O(n^{\frac{\alpha-3}{4}}\sqrt{|S|}). 
\end{align}

Similarly, let $\bar{t}=(t_1,\ldots,t_n)$ be the characteristic vectors of $T$, and let $\bar{t}'=\bar{t}-\left \langle \bar{t},\bar{e}_1 \right \rangle \bar{e_1}$.
Note that 
\begin{align*}
\left \langle A \bar{s}',\bar{t}' \right \rangle & = \left \langle A\bar{s},\bar{t} \right \rangle - \left \langle \bar{s},\bar{e}_1 \right \rangle \left \langle A \bar{e}_1,\bar{t} \right \rangle - \left \langle \bar{t},\bar{e}_1 \right \rangle \left \langle A \bar{s},\bar{e}_1 \right \rangle + \left \langle \bar{s},\bar{e}_1 \right \rangle \left \langle \bar{t},\bar{e}_1 \right \rangle \left \langle A\bar{e}_1,\bar{e}_1 \right \rangle \\
& = e(S,T) - \lambda_1 \left \langle \bar{s},\bar{e}_1 \right \rangle \left \langle \bar{t},\bar{e}_1 \right \rangle.
\end{align*}
Then we have
\[
e(S,T) = \left \langle A \bar{s}',\bar{t}' \right \rangle + \lambda_1 \left \langle \bar{s},\bar{e}_1 \right \rangle \left \langle \bar{t},\bar{e}_1 \right \rangle.
\]
Since $\left \langle \bar{s}',\bar{e_1} \right \rangle =0$ and $\left \langle \bar{t}',\bar{e_1} \right \rangle =0$, it follows that $$|\left \langle A\bar{s}',\bar{t}' \right \rangle |\le |\lambda_2|||s'|| ||t'|| \le |\lambda_2|||s|| ||t||=O(n^{\gamma}\sqrt{|S||T|}).$$
In addition, by equation  \cref{eq: se} and its analogue for $T$, we have
\begin{align*}
\left \langle \bar{s},\bar{e}_1 \right \rangle \left \langle \bar{t},\bar{e}_1 \right \rangle 
& =  \left( \frac{|S|}{\sqrt{n}} + O(n^{\frac{\alpha-3}{4}}\sqrt{|S|})\right) \left( \frac{|T|}{\sqrt{n}} + O(n^{\frac{\alpha-3}{4}}\sqrt{|T|}) \right)
 =\frac{|S||T|}{n} + O(n^{\frac{\alpha-3}{4}}\sqrt{|S||T|}),  
\end{align*}
since $\alpha<3$ and $|S|, |T|\leq n$. Therefore, we deduce
\begin{align*}
\lambda_1 \left \langle \bar{s},\bar{e}_1 \right \rangle \left \langle \bar{t},\bar{e}_1 \right \rangle & = \left( \frac{n}{2}+O(n^\beta) \right)
\left \langle \bar{s},\bar{e}_1 \right \rangle \left \langle \bar{t},\bar{e}_1 \right \rangle
 = \frac{|S||T|}{2} + O(n^{\beta}\sqrt{|S||T|}) + O(n^{\frac{\alpha+1}{4}}\sqrt{|S||T|})
\end{align*}
since $\beta<1$ and $|S|, |T|\leq n$. We conclude that
$$
 e(S,T)=\left \langle A \bar{s}',\bar{t}' \right \rangle +\lambda_1 \left \langle \bar{s},\bar{e}_1 \right \rangle \left \langle \bar{t},\bar{e}_1 \right \rangle=\frac{1}{2}|S||T|+O(n^{\theta}\sqrt{|S||T|}),
$$
where $\theta=\max((1+\alpha)/4, \beta, \gamma).$

Under the assumption of $(2)$, a similar argument shows that
$$
 e(S,T)=\frac{1}{2}|S||T|+O(n^{\theta}\sqrt{|S||T|}),
$$
where
$\theta=\max(\eta, \beta, \gamma).$
\end{proof}

\begin{corollary}\label{cor: expander2}
Let $0<\alpha<2,$ and $0<\beta, \gamma<1$. Let $\mathcal{G}$ be a family of graphs such that the number of vertices of graphs in $\mathcal{G}$ is unbounded.
Assume that all graphs $G$ in $\mathcal{G}$ satisfy the following properties:
$$
e(G)=\frac{n^2}{4}+O(n^\alpha), \quad \lambda_1(G)=\frac{n}{2}+O(n^\beta), \quad\lambda_2(G)=O(n^\gamma),
$$    
where $n=|V(G)|$. Then $\mathcal{G}$ is a family of quasi-random graphs with property $QR(\theta)$, where 
$\theta=\theta(\alpha,\beta, \gamma)=\max((2+\alpha)/4, (3+\beta)/4, \gamma).$
\end{corollary}
\begin{proof}
Let $G \in \mathcal{G}$ and assume that the graph $G$ has all the assumed conditions.
Let $A=A(G)$ be the adjacency matrix of $G$ and let $\bar{v}=(1,1,\ldots,1)^t$. Since $||A\bar{v}|| \le \lambda_1 ||\bar{v}||$, 
$$
||A\bar{v}||^2=\sum_{v}(\deg(v))^2 \le \lambda_1^2 ||\bar{v}||^2
=\frac{n^{3}}{4} + O(n^{\beta +2}).
$$
It follows that
\begin{align*}
\sum_v \left(\deg(v)-\frac{n}{2}\right)^2 
&= \sum_v (\deg(v))^2 - n\Big(\sum_v \deg(v)\Big) +\frac{n^3}{4}\\
& \le \frac{n^3}{4}+ O(n^{\beta +2}) - n \left(\frac{n^2}{2}+O(n^{\alpha}) \right) +\frac{n^3}{4}  
 = O(n^{\alpha+1}+n^{\beta+2}),
\end{align*}    
and thus the result follows immediately from \cref{prop:expander mixing}.
\end{proof}

\begin{remark} 
There are several different versions of the expander mixing lemma for irregular graphs, and we refer to a recent paper \cite{AZ24} and the references therein; see also \cite{C97, KS06}. Among them, it appears that the version by Byrne and Tait \cite{BT23} (building on the work of Krivelevich and Sudakov \cite{KS06}) is closest to the classical expander mixing lemma for regular graphs. However, their version is not always good enough for our purposes, so we prove a version of the expander mixing lemma for the quasi-random graphs above.
\end{remark}

\subsection{Clique number and independence number of quasi-random graphs}

In this section, we prove \cref{thm: 1.1}, a general lower bound on the clique number and independence number of quasi-random graphs. The proposition is essentially known (For example, see 
\cite[Section 6]{T87a} and \cite[Section 4.2]{Thomason99}), except we use the following new bound on Ramsey number $R(k,\ell)$ by Gupta, Ndiaye, Norin, and Wei \cite{GNNW24} to get a slight improvement. 

\begin{theorem}[\cite{GNNW24}]\label{thm: gnnw}
For all positive integers $\ell \leq k$,
$$
R(k, \ell) \leq \exp\left(\rho(\ell / k) k+o(k)\right)\binom{k+\ell}{\ell},
$$
where $\rho(\lambda)=\left(-0.25 \lambda+0.03 \lambda^2+0.08 \lambda^3\right) e^{-\lambda}$.    
\end{theorem}

For our purpose, we deduce the following corollary of Theorem~\ref{thm: gnnw}.

\begin{corollary}
\label{lem:Ramsey}
Let $G$ be a graph with $n$ vertices. Let $r\leq s$ be positive integers such that $n \geq \exp(\rho(r/s)s) \binom{r+s}{r}$. If $\omega(G)\leq r$, then $\omega(G)+\alpha(G) \geq (1-o(1))(r+s)$.
\end{corollary}
\begin{proof}

Consider the function $f(x)$ on $(0,1)$ defined by
$$
f(x)=-x\ln\left(x\right)-\left(1-x\right)\ln\left(1-x\right)+\left(-\frac{0.25x}{1-x}+\frac{0.03x^{2}}{\left(1-x\right)^{2}}+\frac{0.08x^{3}}{\left(1-x\right)^{3}}\right)e^{-\frac{x}{1-x}}\left(1-x\right).
$$
One can check that the function $f$ is strictly increasing on $(0,0.55)$.

Let $r'=\omega(G)$ and let $s'=r+s-r'$. Let $x=r'/(r+s)$ and let $y=r/(r+s)$. Then since $\omega(G)\leq r$ and $r\leq s$, we have $x\leq y\leq \frac{1}{2}$ and thus $f(x)\leq f(y)$. 

By Stirling's approximation, we have
$$
\frac{\ln \left(\exp(\rho(r/s)s) \binom{r+s}{r}\right)}{r+s}=f(y)+o(1), \quad \frac{\ln \left(\exp(\rho(r'/s')s') \binom{r'+s'}{r'}\right)}{r'+s'}=f(x)+o(1).
$$
It follows that
\begin{align*}
n&\geq \exp(\rho(r/s)s) \binom{r+s}{r}=\exp\left((f(y)+o(1)) (r+s)\right)\\
&\geq \exp\left((f(x)+o(1)) (r'+s')\right)=\left(\exp(\rho(r'/s')s') \binom{r'+s'}{r'}\right)^{1+o(1)}\\ & \geq R((1-o(1))s', (1-o(1))r'). 
\end{align*}
The last inequality follows from \cref{thm: gnnw}. Thus, $\alpha(G)\geq (1-o(1))s'$. It follows that $\omega(G)+\alpha(G)\geq (1-o(1))(r+s)$.
\end{proof}

\begin{proposition}\label{thm: 1.1}
Let $\mathcal{G}$ be a family of quasi-random graphs with property $QR(\theta)$ with $\theta \in [1/2,1)$. If $G\in \mathcal{G}$ has $n$ vertices, then we have
$$
\min \{\omega(G),\alpha(G)\} \geq (1-\theta-o(1))\log_2 n,
$$
and
$$
\frac{\omega(G)+\alpha(G)}{2} \geq (1-\theta+\ell-o(1))\log_2 n,
$$
as $n \to \infty$, where $\ell=\ell(\theta) \in (0,1)$ is the unique solution to the equation
\begin{equation}\label{eq:ell}
\ell\log_2 \left( 2+\frac{1-\theta}{\ell} \right) + (\ell+1-\theta)\log_2\left( 1+\frac{\ell}{\ell+1-\theta}\right) +\frac{\ell+1-\theta}{\ln2}\rho\left( \frac{\ell}{\ell+1-\theta}\right)= \theta,    
\end{equation}
where $\rho(\lambda)=\left(-0.25 \lambda+0.03 \lambda^2+0.08 \lambda^3\right) e^{-\lambda}$.
In particular, if $\theta=3/4$, then $\ell \approx 0.3031$ and $\frac{\omega(G)+\alpha(G)}{2} \geq (1-o(1))\log_{3.501} n$; if $\theta=1/2$, then $\ell\approx 0.1436$ and $\frac{\omega(G)+\alpha(G)}{2} \geq (1-o(1))\log_{2.936} n$.
\end{proposition}

\begin{proof}
Let $G \in \mathcal{G}$. 
There is a constant $C>0$ (independent of $G$), such that 
\begin{equation}\label{eq: e(S) edge}
\bigg|e(G[S])-\frac{1}{4}|S|^2\bigg| \leq Cn^{\theta}|S|
\end{equation}
holds for all subsets $S$ of $V(G)$. Without loss of generality, we may assume $C \geq 1$.
By inequality \cref{eq: e(S) edge}, 
\begin{equation}\label{eq: edge for G}
e(G)\geq \frac{n^2}{4}-Cn^{\theta+1}.   
\end{equation}
We first choose a vertex $v_{1}$ with the largest degree in $G$. Then inequality \cref{eq: edge for G} implies that
\begin{equation}\label{eq: Nbh1}
|N(v_1)| \ge \frac{2e(G)}{n}\geq \frac{n}{2} -2Cn^\theta.
\end{equation}
Next, we inductively choose the vertex $v_i$ in the induced subgraph $G_i=G[N(v_1) \cap \cdots \cap N(v_{i-1})]$ with the largest degree of $G_i$ for each $i \geq 2$.
Then, the degree of $v_m$ in the graph $G_m$ is $|N(v_1) \cap N(v_2) \cap \cdots \cap N(v_m)|$, which is at least the average degree $2e(G_m)/|V(G_m)|$ of $G_m$. So, inequality \cref{eq: e(S) edge} implies that
\begin{equation}\label{ineq: nhd}
|N(v_1) \cap N(v_2) \cap \cdots \cap N(v_m)| \ge \frac{1}{2}|N(v_1) \cap N(v_2) \cap \cdots \cap N(v_{m-1})| - 2Cn^{\theta}    
\end{equation}
for $m \ge 2$. Then we use inequality \cref{ineq: nhd} inductively, and this combining with inequality \cref{eq: Nbh1} gives
\begin{align}
|N(v_1) \cap N(v_2) \cap \cdots \cap N(v_m)| & \ge \frac{1}{2^{m-1}}|N(v_1)| -2C\Big( 1+\frac{1}{2} + \cdots + \frac{1}{2^{m-2}} \Big) n^{\theta} \notag\\
& \ge \frac{1}{2^{m-1}}\Big( \frac{n}{2} -2Cn^\theta\Big)-2C\Big( 1+\frac{1}{2} + \cdots + \frac{1}{2^{m-2}} \Big) n^{\theta} \notag\\
& \ge\frac{n}{2^m}   - 4C\Big( 1 - \Big(\frac{1}{2} \Big)^{m}\Big)n^{\theta} \label{eq:lbNvm}.
\end{align}
To obtain an lower bound of $\omega(G)$, 
we choose the largest $m=m(n)$ satisfying 
\[
\frac{n}{2^m} - 4C\Big( 1 - \Big(\frac{1}{2} \Big)^{m}\Big)n^{\theta}>0,
\]
equivalently, 
\[\log_2 \left(\frac{n^{1-\theta}}{4C} +1\right)>m.\]
Therefore, we have $\omega(G) \ge (1-\theta-o(1))\log_2 n$.
For the independence number, note that the complement of $G$ also satisfies inequality \cref{eq: e(S) edge}.
Thus, the same argument shows that $\alpha(G) \geq (1-\theta-o(1))\log_2n$, as desired. This proves the first statement.

Next, we prove the second statement. Since the complement of $G$ also satisfies inequality \cref{eq: e(S) edge}, we may assume that $\omega(G) \leq \alpha(G)$ without loss of generality. Choose $m= 
\lfloor\log_2 \Big(\frac{n^{1-\theta}}{5C}+\frac{4}{5}\Big)\rfloor$, so that $m=(1-\theta-o(1))\log_2n$. Then we have 
\begin{equation}\label{eq: 2m}
2^{m} \le \frac{n^{1-\theta}+4C}{5C}.    
\end{equation}
Then, combining inequality~\cref{eq:lbNvm} with inequality \cref{eq: 2m} shows that
\[
|N(v_1) \cap N(v_2) \cap \cdots \cap N(v_m)| 
\ge n^{\theta}\left( \frac{n^{1-\theta}+4C}{2^m}-4C\right) = Cn^{\theta} \geq n^{\theta}.
\]

Let $\ell=\ell(\theta) \in (0,1)$ be the unique solution to equation~\eqref{eq:ell}.
If $\omega(G) \geq \ell\log_2n+m=(1-\theta+\ell-o(1))\log_2 n$, then we have
$$
\frac{\alpha(G)+\omega(G)}{2} \geq \omega(G) \geq (1-\theta+\ell-o(1))\log_2 n
$$
and we are done. From now on, we assume that $\omega(G)< \ell\log_2n+m$. 

We write $G'=G[N(v_1) \cap N(v_2) \cap \cdots \cap N(v_m)]$. We have shown that there are at least $n^{\theta}$ vertices in $G'$. Set $r=\left \lfloor \ell\log_2 n \right \rfloor$ and $s=r+m$. Next, we claim that if $n$ is sufficiently large, then
\begin{equation}\label{eq:lbG'}    
|V(G')|\geq n^\theta \ge  \exp\left(\rho(r / s) s+o(s)\right)\binom{r+s}{s}.
\end{equation}
Indeed, by Stirling's formula, from the equality $ \binom{r+s}{s}= \binom{2r+m}{r+m} =\frac{(2r+m)!}{r!(r+m)!} $, we deduce
\begin{align*}
\binom{r+s}{s} & \sim \frac{(2r+m)^{2r+m+\frac{1}{2}}}{\sqrt{2\pi}r^{r+\frac{1}{2}}(r+m)^{r+m+\frac{1}{2}}} \notag \\ 
 &= \frac{1}{\sqrt{2\pi r}}\left(2+\frac{m}{r} \right)^r \left(1+\frac{r}{r+m}\right)^{r+m+\frac{1}{2}} \notag \\
 &= \frac{1}{\sqrt{2\pi \ell\log_2n}}\left(2+\frac{m}{\ell\log_2n} \right)^{\ell\log_2n} \left( 1+\frac{\ell\log_2n}{\ell\log_2n+m}\right)^{\ell\log_2n+m+\frac{1}{2}}.  
\end{align*}
It follows that if $n$ is sufficiently large, then 
\begin{align*}
&\frac{1}{\log_2 n} \log_2 \left(\exp\left(\rho(r / s) s+o(s)\right)\binom{r+s}{s}\right)\\
&=\ell\log_2 \left( 2+\frac{1-\theta}{\ell} \right) + (\ell+1-\theta)\log_2\left( 1+\frac{\ell}{\ell+1-\theta}\right)+\frac{\ell+1-\theta}{\ln2}\rho\left( \frac{\ell}{\ell+1-\theta}\right)+o(1)\\ &= \theta+o(1).
\end{align*}
Therefore, if $n$ is sufficiently large, then inequality \cref{eq:lbG'} holds, proving the claim. 

Note that all vertices in $G'$ are adjacent to the clique $\{v_1, v_2, \ldots, v_m\}$ in $G$. 
In particular, $\omega(G)\geq m+\omega(G')$. Since 
$\omega(G)< \ell\log_2n+m=r+m$, we have $\omega(G')<r\leq s$. 
From inequality~\eqref{eq:lbG'} and  
\cref{lem:Ramsey}, we conclude that 
$$
\omega(G)+\alpha(G) \geq m+\omega(G')+\alpha(G') \geq m+(1-o(1))(r+s)=(2-o(1))s=2(1-\theta+\ell-o(1))\log_2 n,
$$
as required.
\end{proof}

\section{Quasi-randomness and clique number of $X_{f,q}$} \label{sec: polynomial}
\subsection{Quasi-randomness of $\mathcal{X}_d$}\label{subsec: quasi-random} 
The purpose of this subsection is to show that the family $\mathcal{X}_d$ is a family of quasi-random graphs with property $QR(3/4)$ (\cref{thm: 3/4 intro}). To achieve our purpose, we begin by introducing the celebrated Hoffman-Wielandt inequality for the perturbation of eigenvalues.

\begin{lemma}[Hoffman-Wielandt inequality \cite{HW53}]\label{lem: HW}
Let $A,B \in M_{n,n}(\mathbb{C})$ be normal with the corresponding eigenvalues, $\lambda_{1}(A),\ldots,\lambda_{n}(A)$, and $\lambda_{1}(B),\ldots,\lambda_{n}(B)$. Then
\[\min_{\sigma \in S_{n}}\sum_{i=1}^{n}|\lambda_{i}(A)-\lambda_{\sigma(i)}(B)|^{2} \le ||A-B ||^{2}.\]
Here, $S_{n}$ is the permutation group of $\{1,\ldots, n\}$, and $||A||^{2}=\sum_{1\le i,j \le n}|a_{i,j}|^{2}$ for any $A=(a_{i,j}) \in M_{n,n}(\mathbb{C})$.
\end{lemma}

To apply \cref{lem: HW}, we need to estimate the entries of the square of the adjacency matrix of the graph $X_{f,q}$.

\begin{proposition}\label{prop:A^2}
Let $f\in \F_q[x,y]$ be an admissible polynomial of degree $d$. Let $A$ be the adjacency matrix of the graph $X_{f,q}$. Then all but $O_d(1)$ of diagonal entries of $A^2$ are $q/2+O_d(\sqrt{q})$, and all but $O_d(q)$ of off-diagonal entries of $A^2$ are $q/4+O_d(\sqrt{q})$.     
\end{proposition}
\begin{proof}
We first consider the diagonal entries of $A^2$, which correspond to the degree of vertices in $X_{f,q}$. By \cref{cor:badu}, the number of $u$ such that $f(x,u)$ is a constant multiple of a square of a polynomial is $O(d^2)$. On the other hand, if $f(x,u)$ is not a constant multiple of a square of a polynomial, then by \cref{Weil}, the degree of $u$ is 
\[\frac{1}{2}\sum_{x \in \F_q } \Big( \chi(f(x,u)) + \delta_{0,f(x,u)} + 1\Big) = \frac{q}{2} + O_d(\sqrt{q}),\]
as required. 

Next, we consider off-diagonal entries of $A^2$. \cref{cor:badu} and \cref{cor:badpairs} imply that there are at most $O(d^4+d^2q)=O_d(q)$ pairs $(u,v)$ such that at least one of $f(x,u), f(x,v)$, and $f(x,u)f(x,v)$ is a constant multiple of a square of a polynomial. Thus, it suffices to show if neither $f(x,u), f(x,v)$, nor $f(x,u)f(x,v)$ is a constant multiple of a square of a polynomial, then the entry $(u,v)$ of $A^2$ is $q/4+O_d(\sqrt{q})$. Let $(u,v)$ be such a pair that neither $f(x,u), f(x,v)$, nor $f(x,u)f(x,v)$ is a constant multiple of a square of a polynomial. Let $\chi$ be the quadratic character of $\F_q$, and let $\eta(x)=\chi(x)+\delta_{0,x}$ so that $\eta(x)=1$ when $x$ is a square in $\F_q$ (including $0$), and $\eta(x)=-1$ when $x$ is a non-square in $\F_q$. Then the entry $(u,v)$ of $A^{2}$ can be written as 
\begin{align}
\label{char-cal}
&\frac{1}{4}\sum_{x \in \mathbb{F}_{q}} (\eta(f(x,u))+1)(\eta(f(x,v))+1) \notag\\
&=\frac{1}{4}\sum_{x \in \mathbb{F}_{q}} (\chi(f(x,u))+1)(\chi(f(x,v))+1)+\frac{1}{4}\sum_{x \in \mathbb{F}_{q}} \delta_{0,f(x,u)}(\chi(f(x,v))+1) \notag\\
&+\frac{1}{4}\sum_{x \in \mathbb{F}_{q}} \delta_{0,f(x,v)}(\chi(f(x,u))+1)+\frac{1}{4}\sum_{x \in \mathbb{F}_{q}} \delta_{0,f(x,u)}\cdot\delta_{0,f(x,v)}.
\end{align}
We note that
\[\frac{1}{4}\sum_{x \in \mathbb{F}_{q}} \delta_{0,f(x,u)}\cdot\delta_{0,f(x,v)}=\frac{1}{4}\#\{x\in \F_{q}~|~f(x,u)=f(x,v)=0\}=O_d(1),\]
\begin{align*}  
\frac{1}{4}\sum_{x \in \mathbb{F}_{q}} \delta_{0,f(x,u)}(\chi(f(x,v))+1) =\frac{1}{4}\sum_{\substack{ x \in \mathbb{F}_{q}\\ f(x,u)=0}}(\chi(f(x,v))+1) 
& = O_d(1).
\end{align*}
It follows that the off-diagonal entry $(u,v)$ of $A^{2}$ is
\begin{equation*}
\frac{q}{4}+\frac{1}{4}\sum_{x \in \mathbb{F}_{q}}\chi(f(x,u))+\frac{1}{4}\sum_{x \in \mathbb{F}_{q}}\chi(f(x,v))+\frac{1}{4}\sum_{x \in \mathbb{F}_{q}}\chi(f(x,u)f(x,v))+ O_d(1).
\end{equation*}
Now Weil's bound implies the entry $(u,v)$ is $q/4+O_d(\sqrt{q})$. 
\end{proof}

We now have all the ingredients to prove \cref{thm: 3/4 intro}.

\begin{proof}[Proof of \textup{\cref{thm: 3/4 intro}}]
Let $X_{f,q} \in \mathcal{X}_d$. \cref{prop:A^2} implies that all but at most $O_d(1)$ of vertices have degree $q/2+O_d(\sqrt{q})$. In particular, we have
\begin{equation}\label{eq: deg1}
    \sum_v \Big(\deg(v)-\frac{q}{2}\Big)^2 = O_d(q^2)
\end{equation}
and 
\begin{equation}\label{eq: edge 3/4}
     e(X_{f,q})=\frac{1}{2}\sum_{u \in \F_q} \deg (u)=\frac{q^2}{4}+O_d(q^{3/2}).
\end{equation}

Let $A$ be the adjacency matrix of the graph $X_{f,q}$.  We approximate $A^2$ by $B=\frac{q}{2}I+\frac{q}{4}(J-I)=\frac{q}{4}(I+J)$, where $I$ is the identity matrix, and $J$ is the all one matrix (both $I$ and $J$ are $q \times q$ matrices). Note that the eigenvalues of $\frac{q}{4}(I+J)$ are $\frac{q}{4}$ (with multiplicity $q-1$) and $\frac{(q+1)q}{4}$ (with multiplicity $1$). By \cref{prop:A^2}, all but at most $O_d(q)$ entries of $A^2-B$ are bounded by $O_d(\sqrt{q})$. Also, note that all entries of $A^2-B$ are trivially bounded by $q$. Let $\lambda_i=\lambda_i(X_{f,q})$. By Hoffman-Wielandt inequality, there exists $1 \leq i \leq q$, such that
$$
\left(\lambda_i^2-\frac{(q+1)q}{4}\right)^2+ \sum_{1 \leq j \leq q, j \neq i} \left(\lambda_j^2-\frac{q}{4}\right)^2 \leq ||A^2-B||^2\leq q^2 O_d(q)+ O_d(q)  q^2=O_d(q^3).
$$
We claim that $i=1$. Indeed, since $\lambda_1$ is at least the average degree of the graph $X_{f,q}$, equation \cref{eq: edge 3/4} implies $\lambda_1\geq \frac{q}{2}+O_d(q^{1/2})$. If $i \neq 1$, then we would have
$(\lambda_1^2-\frac{q}{4})^2=O_d(q^3)$, which is impossible.  
Therefore, we have
$$
\left(\lambda_1^2-\frac{(q+1)q}{4}\right)^2+\left(\lambda_2^2-\frac{q}{4}\right)^2
\leq \left(\lambda_1^2-\frac{(q+1)q}{4}\right)^2+ \sum_{j=2}^q \left(\lambda_j^2-\frac{q}{4}\right)^2
\leq O_d(q^3).
$$
It follows that $\lambda_1=\frac{q}{2}+O_d(q^{1/2})$ and $\lambda_2=O_d(q^{3/4})$. 
This, together with equation \cref{eq: deg1}, satisfies the assumptions in \cref{prop:expander mixing}.
Therefore, we conclude that $\mathcal{X}_d$ is a family of quasi-random graphs with property $QR(3/4)$. 
\end{proof}

\begin{remark}
Note that Paley graphs are only defined over $\F_q$ with $q \equiv 1 \pmod 4$. Thomason \cite{T16} constructed a (undirected) Paley-like graph $G_q$ for a finite field $\F_q$ with $q=4^n$: the vertex set of $G_q$ is $PG(1, q)$, and the edge set of $G_q$ is defined via the trace map $\Tr_{\F_q/\F_2}$. He showed that the graph $G_q$ is a self-complementary Cayley graph, $G_q$ is $q/2$-regular, and the co-degree of each pair is $q/4+O(\sqrt{q})$. Using these properties, he deduced that $G_q$ is $(1/2, q^{3/4})$-jumbled. 
By slightly modifying the proof of \cref{thm: 3/4 intro}, we can show $\mathcal{G}=\{G_q: q=4^n\}$ is a family of quasi-random graphs with property $QR(3/4)$. 
Thus, it follows from \cref{thm: 1.1} that $\omega(G_q) \geq (1-o(1))\log_{3.501} q$, as $q=4^n \to \infty$. 
\end{remark}

\subsection{Quasi-randomness of $\mathcal{H}_d$}
In this subsection, we show that $\mathcal{H}_d$ is a family of quasi-random graphs with property $QR(1/2)$ (\cref{thm: construction intro}) and prove \cref{thm: lower bound of homo}. We begin with the following lemma.

\begin{lemma}\label{lem:fcomposeg}
Let $\mathcal{G}$ be a family of graphs such that each graph in $\mathcal{G}$ is of the form $X_{f,q}$, and let $d \geq 1$. Let $\mathcal{F}_d$ be the family consisting of all graphs of the form $X_{h,q}$, where $g \in \F_q[x]$ is a polynomial with degree $d$, $h(x,y)=f(g(x),g(y))$, and $X_{f,q} \in \mathcal{G}$. If $\mathcal{G}$ is a family of quasi-random graphs with property $QR(1/2)$, then so is $\mathcal{F}_d$.
\end{lemma}

\begin{proof}
Let $X_{f,q} \in \mathcal{G}$. Let $g \in \F_q[x]$ be a polynomial with degree $d$ and let $h(x,y)=f(g(x),g(y))$ so that $X_{h,q} \in \mathcal{F}_d$. We can partition $\F_q$ into the union of $d$ disjoint subsets $V_1,V_2, \ldots, V_d$, such that $g$ is injective on $V_i$ for each $1 \leq i \leq d$. Now, given two subsets $S, T$ of $\F_q$, set $S_i=S \cap V_i$ and $T_i=T \cap V_i$. Note that if $g(u) \neq g(v)$, then the two vertices $u,v$ are adjacent in $X_{h,q}$ if and only if the two vertices $g(u),g(v)$ are adjacent in $X_{f,q}$. Note that the number of pairs of vertices $u\in S$ and $v \in T$ such that $u,v$ are adjacent in $X_{h,q}$ and $g(u)=g(v)$ is at most $d\min \{|S|, |T|\} \leq d\sqrt{|S||T|}$. Based on the above observations and the assumption that $\mathcal{G}$ is a family of quasi-random graphs with property $QR(1/2)$, in the graph $X_{h,q}$, we have 
\begin{align*}
e(S,T)
&=\sum_{i,j=1}^{d} e(S_i, T_j)= O(d\sqrt{|S||T|})+\sum_{i,j=1}^{d} e_{X_{f,q}}(g(S_i), g(T_j))\\
&=O(d\sqrt{|S||T|})+\sum_{i,j=1}^{d} \bigg(\frac{|S_i||T_j|}{2}+O(\sqrt{q|S_i||T_j|})\bigg)
=\frac{|S||T|}{2}+ O_d(\sqrt{q}\sqrt{|S||T|}).
\end{align*}
This proves that $\mathcal{F}_d$ is a family of quasi-random graphs with property $QR(1/2)$, as required.
\end{proof}

Next, under some additional assumptions, we modify the proof of \cref{thm: 3/4 intro} to show a stronger quasi-randomness of $X_{f,q}$.

\begin{proposition}\label{prop: 1/2}
Let $d \geq 1$ and let $\mathcal{G}_d$ be an infinite subfamily of $\mathcal{X}_d$. Assume that for each $X_{f,q} \in \mathcal{G}_d$, there is a subset $B$ of $\F_q$ with $|B|=O(1)$, such that the entry $(u,v)$ of $A(X_{f,q})^2$ are $q/4+O(1)$ for all $u,v \in \F_q \setminus B$ with $u \neq v$. Then $\mathcal{G}_d$ is a family of quasi-random graphs with property $QR(1/2)$.
\end{proposition}

\begin{proof}
Let $X_{f,q} \in \mathcal{G}_d$. By assumption, there exists a subset $B$ of $\F_q$ with $|B|=O(1)$, such that the entry $(u,v)$ of $A(X_{f,q})^2$ are $q/4+O(1)$ for all $u,v \in \F_q \setminus B$  with $u \neq v$. Let $B' \subset \F_q$ be the subset of vertices of $X_{f,q}$ with degree not equal to $q/2+O_d(\sqrt{q})$. By \cref{prop:A^2}, we have $|B'|=O_d(1)$. Let $V'=\F_q \setminus B \setminus B'$ and let $\widetilde{q}=|V'|$; then $\widetilde{q}=q-O_d(1)$. 
Let $\widetilde{X}=\widetilde{X}_{f,q}$ be the subgraph of $X_{f,q}$ induced by $V'$. Let $\widetilde{A}$ be the adjacency matrix of $\widetilde{X}$. Then we have 
\begin{itemize}
    \item among the diagonal entries of $\widetilde{A}^2$, all of them are $\widetilde{q}/2+O_d(\sqrt{\widetilde{q}})$;
    \item among the off-diagonal entries of $\widetilde{A}^2$, all of them are $\widetilde{q}/4+O_d(1)$.
\end{itemize}

Next, we apply a similar argument as in the proof of \cref{thm: 3/4 intro} to estimate $\widetilde{\lambda_1}=\widetilde{\lambda_1}(\widetilde{X})$ and $\widetilde{\lambda_2}=\widetilde{\lambda_2}(\widetilde{X})$. Note that we have $\sum_{v\in \widetilde{X}} \Big(\deg (v)-\frac{\widetilde{q}}{2}\Big)^2 = O_d(\widetilde{q}^2)$ and $e(\widetilde{X})=\frac{\widetilde{q}^2}{4}+O_d(\widetilde{q}^{3/2})$. It follows that $\widetilde{\lambda_1} \geq \frac{\widetilde{q}}{2}+O_d(\sqrt{\widetilde{q}})$. Let us approximate $\widetilde{A}^2$ by $\widetilde{B}=\frac{\widetilde{q}}{2}I+\frac{\widetilde{q}}{4}(J-I)=\frac{\widetilde{q}}{4}(I+J)$, where $I$ is the identity matrix, and $J$ is the all one matrix (both $I$ and $J$ are $\widetilde{q} \times \widetilde{q}$ matrices). Note that the eigenvalues of $\frac{\widetilde{q}}{4}(I+J)$ are $\frac{\widetilde{q}}{4}$ (with multiplicity $\widetilde{q}-1$) and $\frac{(\widetilde{q}+1)\widetilde{q}}{4}$ (with multiplicity $1$). The assumptions on the entries of $\widetilde{A}^2$ guarantee that $||\widetilde{A}^2-\widetilde{B}||^2=O_d(\widetilde{q}^2)$. Using Hoffman-Wielandt inequality, we deduce that 
$$
\left(\widetilde{\lambda_1}^2-\frac{(\widetilde{q}+1)\widetilde{q}}{4}\right)^2+\left(\widetilde{\lambda_2}^2-\frac{\widetilde{q}}{4}\right)^2=O_d(\widetilde{q}^2)
$$
and thus $\widetilde{\lambda_1}=\frac{\widetilde{q}}{2}+O_d(1)$, $\widetilde{\lambda_2}=O_d(\sqrt{\widetilde{q}})$. It then follows from \cref{prop:expander mixing} that $\widetilde{\mathcal{G}}_d:=\{\widetilde{X}_{f,q}: X_{f,q}\in \mathcal{G}_d\}$ is a family of quasi-random graphs with property $QR(1/2)$.

Now we go back to the original family $\mathcal{G}_d$. Note that for each $X_{f,q} \in \mathcal{G}_d$, the graph $\widetilde{X}_{f,q}$ is obtained from $X_{f,q}$ by deleting $O_d(1)$ vertices. It is straightforward to verify that $\mathcal{G}_d$ remains a family of quasi-random graphs with property $QR(1/2)$.
\end{proof}

As a corollary, the admissible polynomials of bidegree $(1,1)$ induce a family of quasi-random graphs with property $QR(1/2)$.

\begin{corollary}\label{cor:deg1}
Let $\mathcal{X}_{1,1}$ be the family consisting of all graphs $X_{f,q}$ and $X_{xyf,q}$, where $f\in\mathbb{F}_q[x,y]$ is an admissible polynomial with bidegree $(1,1)$. Then $\mathcal{X}_{1,1}$ is a family of quasi-random graphs with property $QR(1/2)$. 
\end{corollary}

\begin{proof}
Let $f(x,y)=d+ax+by+cxy \in \F_q[x,y]$ be an admissible polynomial for some $a,b,c,d\in \F_q$. Then the coefficients satisfy $ab-cd\neq 0$. Indeed, if $d \neq 0$, then we have the factorization of the form $f(x,y)=\frac{1}{d}(ax+d)(by+d)$, violating the assumption that $f$ is admissible. If $d=0$, it is easy to verify that $f$ is not admissible.  

Observe that for all but at most one $u \in \F_q$, the polynomial $f(x,u)\in\mathbb{F}_q[x]$ has degree $1$. Let $u,v \in \F_q$ with $u \neq v$. Then $f(x,u)f(x,v)$ has degree at most $2$, and $f(x,u)f(x,v)$ is not a constant. Suppose that the polynomial $f(x,u)f(x,v)$ is a constant multiple of a square of a polynomial, then we would have
\begin{align*}
   f(x,u)f(x,v) 
    &=(cu+a)(cv+a)\Big(x+\frac{d+bu}{cu+a}\Big)\Big(x+\frac{d+bv}{cv+a}\Big),
\end{align*}
and the constant terms of the linear factors should be the same to have a double root; this forces $\frac{d+bu}{cu+a}=\frac{d+bv}{cv+a}$, which happens only if $ab-cd=0$ as $ u\neq v$, a contradiction. Let $\chi$ be the quadratic character of $\F_q$. Then \cite[Theorem 5.48]{LN97} implies
$$
\bigg|\sum_{x \in \mathbb{F}_{q}^*}\chi(f(x,u)f(x,v))\bigg|\le 2.
$$
Hence, from equation~\eqref{char-cal}, it is easy to verify that the off-diagonal entry $(u,v)$ of $A(X_{f,q})^2$ is $q/4+O(1)$, unless $f(x,u)$ or $f(x,v)$ is a constant. 

Next, we consider the graph $X_{xyf,q}$. Let $u,v \in \F_q$ with $u \neq v$. We have 
\[\sum_{x \in \mathbb{F}_{q}} \chi\big(xuf(x,u) \cdot xvf(x,v)\big)=\chi(uv)\sum_{x \in \mathbb{F}_{q}^*}\chi(f(x,u)f(x,v))=O(1).\]
Using a similar argument, the off-diagonal entry $(u,v)$ of $A(X_{xyf,q})^2$ is $q/4+O(1)$, unless $uv=0$, or one of $f(x,u)$ and $f(x,v)$ is a constant. 
Hence \cref{prop: 1/2} implies the desired result.
\end{proof}

We conclude the paper with the proof of \cref{thm: construction intro}, \cref{cor: xfq}, and \cref{thm: lower bound of homo}.

\begin{proof}[Proof of \textup{\cref{thm: construction intro}}]
This follows from \cref{lem:fcomposeg}, \cref{cor: xfq} and \cref{cor:deg1}.    
\end{proof}

\begin{proof}[Proof of \textup{\cref{cor: xfq}}]
This follows from \cref{thm: 3/4 intro}, \cref{thm: construction intro} and \cref{thm: 1.1}.    
\end{proof}

\begin{proof}[Proof of \textup{\cref{thm: lower bound of homo}}]
Let $I$ be a maximum independent set in $X_{f,q}$ and let $r$ be a non-square in $\F_q$. Then $rI$ is a clique in $X_{f,q}$. Indeed, since the degree of $f$ is $d$ and $d$ is odd, $f(rx, ry)=r^df(x,y)$ is a square for each $x,y \in I$. It follows that $\omega(X_{f,q}) \geq \alpha(X_{f,q})$. 
Note that we have $\max \{\omega(X_{f,q}), \alpha(X_{f,q})\} \geq (1-o(1))\log_{3.501} q$ by \cref{cor: xfq}. Therefore, $\omega(X_{f,q})\geq (1-o(1)) \log_{3.501} q$. 
Similarly, if $X_{f,q} \in \mathcal{H}_d$, we have $\omega(X_{f,q})\geq (1-o(1)) \log_{2.936} q$.  
\end{proof}

\section*{Acknowledgments}
The authors thank Hong Liu, J\'ozsef Solymosi, and Andrew Thomason for helpful discussions. S. Kim is grateful to the Max Planck Institute for Mathematics in Bonn for its hospitality and financial support. C.H. Yip was supported in part by an NSERC fellowship.  S. Yoo was supported by the KIAS Individual Grant (CG082701) at the Korea Institute for Advanced Study and by the Institute for Basic Science (IBS-R029-C1). 

\medskip

\bibliographystyle{abbrv}
\bibliography{references}

\end{document}